\documentclass[a4paper,12pt,reqno]{amsart}
\usepackage{amssymb,amsmath,hyperref}
\usepackage{graphicx}
\usepackage[T2A]{fontenc}
\usepackage{verbatim}
\usepackage{clrscode3e}
\usepackage{bbold}
\usepackage{mathtools}
\usepackage{centernot}

\usepackage{comment}
\title[A new Andrews--Crandall-type identity]{A new Andrews--Crandall-type identity and the number of integer solutions to $x^2+2y^2+2z^2=n$}
\author{Mariia Dospolova}
\author{Ekaterina Kochetkova}
\author{Eric T. Mortenson}
\address{Department of Mathematics and Computer Science, Saint Petersburg State University, Saint Petersburg, 199034, Russia}
\email{dospolova.maria@yandex.ru}
\email{ekaterinakochetkova20@gmail.com }
\email{etmortenson@gmail.com}

\renewcommand\theta{\vartheta}

\newtheorem{theorem}{Theorem}
\newtheorem{lemma}[theorem]{Lemma}
\newtheorem{corollary}[theorem]{Corollary}
\newtheorem{proposition}[theorem]{Proposition}
\theoremstyle{definition}
\newtheorem{definition}[theorem]{Definition}
\newtheorem{remark}[theorem]{Remark}

\numberwithin{theorem}{section} 
\numberwithin{equation}{section}

\makeatletter
\@namedef{subjclassname@2020}{\textup{2020} Mathematics Subject Classification}
\makeatother

\begin{document}

\date{6 July 2023}



\begin{abstract}
     Using a higher-dimensional analog of an identity known to Kronecker, we discover a new Andrews--Crandall-type identity and use it to count the number of integer solutions to $x^2+2y^2+2z^2=n$.
\end{abstract}

\maketitle

\section{Introduction}

It has long been known that an identity known to Kronecker gives as special cases many classical number theory results on the representations of a number as a sum of squares \cite{Warn}.  For example we can obtain Lagrange's theorem that every positive integer can be written as the sum of four squares, and we can obtain Fermat's theorem that an odd prime number $p$ can be written as a the sum of two squares if an only if $p\equiv 1 \pmod 4$.  Are there higher-dimensional generalizations of Kronecker's identity?  

In \cite{Mo2017A}, Mortenson found a higher-dimensional analog of Kronecker's identity, and among its specializations \cite{Mo2017B} are an identity that yields Gauss's Eureka theorem that every positive integer can be written as a sum of three triangular numbers, and the Andrews--Crandall identity which counts the number of representations of $n$ of the form 
\begin{equation}
n=x^2+y^2+z^2.\label{equation:3squares}
\end{equation}  
Counting solutions to (\ref{equation:3squares}) has been well-studied, see for example \cite{G}.  Are there other specializations of the higher-dimensional analog of Kronecker's identity that yield number theoretic results?

In this paper we obtain another specialization of the higher-dimensional generalization of Kronecker's identity, where the new specialization resembles the Andrews--Crandall identity but instead counts the number of representations of $n$ of the form 
\begin{equation}
n=x^2+2y^2+2z^2.\label{equation:3squares-Alt}
\end{equation}
One easily sees that the (\ref{equation:3squares}) and (\ref{equation:3squares-Alt}) are related.  Here one notes that
\begin{align*}
x^2+2y^2+2z^2=x^2+(y-z)^2+(y+z)^2=x^2+u^2+v^2, \ u\equiv v \pmod 2.
\end{align*}
However, our methods and approach to studying (\ref{equation:3squares-Alt}) are new.  Here, we will improve the methods found in \cite{Mo2017B}, and we will develop the notion of an Andrews--Crandall-type identity.   Along the way, we will frequently need to count solutions to variations of
\begin{equation}
xy+yz+xz=n, \ \textup{where} \ x,y,z\ge 1,\label{equation:HLM}
\end{equation}
a problem that appears to have been initiated by Hermite, Liouville, and Mordell \cite{Mord1}.

\subsection{Background: an identity known to Kronecker} Let $q$ be a nonzero complex number with $|q|<1$ and define $\mathbb{C}^*:=\mathbb{C}-\{0\}$.  Recall
\begin{gather*}
(x)_n=(x;q)_n:=\prod_{i=0}^{n-1}(1-q^ix), \ \ (x)_{\infty}=(x;q)_{\infty}:=\prod_{i\ge 0}(1-q^ix),\\
(x_1,x_2,\dots,x_n;q)_{\infty}:=(x_1)_{\infty}(x_2)_{\infty}\cdots (x_n)_{\infty}.
\end{gather*}
The following identity was known to Kronecker \cite{Kron1}, \cite[pp. 309--318]{Kron2}, and \cite[pp. 70--71]{Weil}; however, Kronecker's identity is also a special case of Ramanujan's ${}_{1}\psi_{1}$-summation \cite{Warn}.  We provide the symmetric form.  For  $x,y\in \mathbb{C}^*$ where $|q|<|x|<1$ and $|q|<|y|<1$
\begin{equation}
\Big ( \sum_{r,s\ge 0}-\sum_{r,s<0} \Big)q^{rs}x^ry^s=\frac{(q)_{\infty}^2(xy,q/xy;q)_{\infty}}{(x,q/x,y,q/y;q)_{\infty}}.\label{equation:kronecker}
\end{equation}
Kronecker's identity yields many classical number theoretic results as special cases, but first we rewrite the identity in a more useful form.

Define $r_s(n)$ to be the number of representations of $n$ as a sum of $s$ squares and define the generating function
\begin{equation*}
\sum_{n\ge 0 }r_s(n)(-q)^n
=\Big ( \sum_{m\in \mathbb{Z} }(-1)^mq^{m^2}\Big )^s
=\Big ((q)_{\infty}/(-q)_{\infty} \Big )^s,
\end{equation*}
where the last equality follows from Jacobi's triple product identity:
\begin{equation}
(x)_{\infty}(q/x)_{\infty}(q)_{\infty}=\sum_{n=-\infty}^{\infty}(-1)^nq^{\binom{n}{2}}x^n,\label{equation:JTPid}
\end{equation}
Multiply both sides of (\ref{equation:kronecker}) by $(1-x)(1-y)/(1-xy)$ and rewrite to have
\begin{equation}
1+\frac{(1-x)(1-y)}{1-xy}\sum_{r,s\ge 1}q^{rs}(x^ry^s-x^{-r}y^{-s})
=\frac{(xyq)_{\infty}(q/xy)_{\infty}(q)_{\infty}^2}{(xq)_{\infty}(q/x)_{\infty}(yq)_{\infty}(q/y)_{\infty}}.\label{equation:kronecker-alt}
\end{equation}
Taking the limit $x=y\rightarrow -1$ of (\ref{equation:kronecker-alt}) yields Jacobi's four-square theorem
\begin{equation*}
r_4(n)=8\sum_{d\mathop{\raisebox{-2pt}{$\centernot\vdots$}}4, n\mathop{\raisebox{-2pt}{\vdots}}d}d,
\end{equation*}
which implies Lagrange's theorem that every positive integer is a sum of four squares.  Taking the limit $x$, $y^2\rightarrow -1$ of (\ref{equation:kronecker-alt}) gives a result of Gauss and Lagrange
\begin{equation}
r_2(n)=4(d_1(n)-d_3(n)),\label{equation:fermat}
\end{equation}
where $d_k(n)$ is the number of divisors of $n$ congruent to $k$ modulo $4$.  Identity (\ref{equation:fermat}) implies Fermat's two-square theorem.

\subsection{Background: a higher-dimensional generalization of Kronecker's identity} It turns out that there is a higher-order analog of (\ref{equation:kronecker}).  We present the symmetrized version:
\begin{theorem}\cite[Corollary $1.2$]{Mo2017A} \label{theorem:result} For  $x,y,z\in \mathbb{C}^*$ where $|q|<|x|<1$, $|q|<|y|<1$,  and $|q|<|z|<1$,
{\allowdisplaybreaks \begin{align}
\Big ( \sum_{r,s,t \ge 0}&+\sum_{r,s,t<0} \Big)q^{rs+rt+st}x^ry^sz^t\label{equation:thm-result}\\
&=  \frac{(yz,q^2/yz;q^2)_{\infty}}{(y,z,q/y,q/z;q)_{\infty}}
\frac{(q;q)_{\infty}^2}{(q^2;q^2)_{\infty}} \sum_{k\in \mathbb{Z}}\frac{(-1)^kq^{k^2}(yz)^k}{1+q^{2k}x}\notag \\
&\ \ \ \ \ + \frac{(xy,q^2/xy;q^2)_{\infty}}{(x,y,q/x,q/y;q)_{\infty}}
 \frac{(q;q)_{\infty}^2}{(q^2;q^2)_{\infty}}\sum_{k\in \mathbb{Z}} \frac{(-1)^kq^{k^2}(xy)^k}{1+q^{2k}z}\notag \\
&\ \ \ \ \ -2  \frac{(q^2;q^2)_{\infty}^3}{(x,y,z,q/x,q/y,q/z;q)_{\infty}} \frac{(xy,xz,yz,q^2/xy,q^2/xz,q^2/yz;q^2)_{\infty}}{(-x,-y,-z,-q^2/x,-q^2/y,-q^2/z;q^2)_{\infty}}\notag \\
&\ \ \ \ \ + \frac{(xz,q^2/xz;q^2)_{\infty}}{(x,z,q/x,q/z;q)_{\infty}}
 \frac{(q;q)_{\infty}^2}{(q^2;q^2)_{\infty}}  \sum_{k\in \mathbb{Z}} \frac{(-1)^kq^{k^2}(xz)^k}{1+q^{2k}y}.\notag 
\end{align}}%
\end{theorem}

We recall the first well-known result that can be deduced from Theorem \ref{theorem:result}.  Let us define $r_{3\Delta}(n)$ to be the number of representations of $n$ as a sum of three triangular numbers, where a triangular number is a number of the form $k(k-1)/2$.  We write the generating function as
\begin{equation*}
\sum_{n=0}^{\infty}r_{3\Delta}(n)q^n:=\Big ( \sum_{n=0}^{\infty}q^{\binom{n+1}{2}}\Big )^3=\Big ( \frac{(q^2;q^2)_{\infty}^2}{(q;q)_{\infty}}\Big )^3, 
\end{equation*}
where we have used Jacobi's triple product identity (\ref{equation:JTPid}).  Gauss discovered that every positive integer is representable as a sum of three triangular numbers, \cite[Art. $293$]{G}.  

If we specialize Theorem \ref{theorem:result} with $x=y=z$ and then make the substitutions: $q\mapsto q^2$, $x\mapsto q$, we obtain the following identity, which gives as an immediate corollary Gauss's result on triangular numbers. 
\begin{theorem}\cite[Theorem $1.7$]{Mo2017B}\label{theorem:EYPHKA} We have
\begin{equation*}
\sum_{n=0}^{\infty}r_{3\Delta}(n)q^n=1+3\sum_{r\ge 1}^{\infty}q^r+3\sum_{r,s \ge 1}q^{2rs+r+s} 
+\Big ( \sum_{r,s,t > 0}+\sum_{r,s,t < 0}\Big )q^{2rs+2rt+2st+r+s+t}.
\end{equation*}
\end{theorem}
\noindent Gauss's EYPHKA theorem is also a corollary of an identity of Andrews \cite[$(5.17)$]{A}.

In \cite[Section $3$]{Mo2017B}, Mortenson also showed that specializing Theorem \ref{theorem:result} with $x=y=z$, $x\mapsto -1$, gives a new proof of an identity found by Crandall in his work on the Madelung constant.  Crandall \cite{Cr} has also pointed out that his identity may be obtained from Andrews's identity \cite[$(5.16)$]{A}.  The Andrews--Crandall identity reads
\begin{theorem} \cite[$(6.2)$]{Cr} \label{theorem:ac-identity} For positive integers $n$, we have
\begin{equation*}
r_3(n)=6(-1)^{n+1}\sum_{\substack{r,s \ge 1\\rs=n}}(-1)^{r+s} +4(-1)^{n+1}\sum_{\substack{r,s,t \ge 1\\rs+rt+st=n}}(-1)^{r+s+t}.
\end{equation*}
\end{theorem}
Mortenson \cite{Mo2017B} then used the Andrews--Crandall identity to give a new proof of
\begin{theorem}[Gauss] \label{theorem:gauss-withsquare} For positive integers $n$, we have
\begin{align*}
r_3(n)&=12H(4n), & \textup{for }n\equiv 1,\ 2,\ 5, \textup{ or } 6 \pmod 8,\\
r_3(n)&=24H(n), &  \textup{for }n\equiv 3 \pmod 8,\\
r_3(n)&=0, &  \textup{for }n\equiv 7 \pmod 8,\\
r_3(n)&=r_3(n/4),  &  \textup{for }n\equiv 0 \pmod 4,
\end{align*}
where $H(d)$ is the Hurwitz class number of discriminant $d$.
\end{theorem}

\noindent A consequence of Gauss's theorem is the celebrated local to global principle:

\smallskip
\begin{center}
{ \em Legendre/Gauss $(1800)$: $r_3(n)>0$ if and only if $n\ne 4^a(8b+7)$.}  
\end{center}

\subsection{The main results}
We will demonstrate a specialization of Theorem \ref{theorem:result} that gives a new Andrews--Crandall-type identity.
\begin{theorem}\label{theorem:result-DKM} We have
\begin{align}
1& -
4 \sum_{r,s\ge 1}q^{2rs}(-1)^{r+s}
 - 2\sum_{s,t\ge 1}q^{4st}(-1)^{s+t}-2 \sum_{s,t\ge 1}q^{(2s-1)(2t-1)}(-1)^{s+t}  \label{relation_DKM}\\
 & \qquad 
 -4\sum_{r,s,t\ge 1} q^{2rs+2rt+4st}(-1)^{r+s+t}
  -4 \sum_{r,s,t\ge 1}q^{r(2s-1)+r(2t-1)+(2s-1)(2t-1)}(-1)^{r+s+t}\notag \\
 &=\left ( \frac{(q;q)_{\infty}}{(-q;q)_{\infty}}\right ) 
\left( \frac{(q^2;q^2)_{\infty}}{(-q^2;q^2)_{\infty}}\right )^2=:\sum_{n=0}^{\infty}a(n)q^{n}.\notag 
\end{align}
\end{theorem}
It is straight forward to write Theorem \ref{theorem:result-DKM} in a form closer to the Andrews--Crandall identity:
\begin{corollary} \label{cor:result-DKM}
For positive integers $n$, we have
\begin{align}
\label{DKM_even}
    a(n) &= -4\sum_{\substack{r,s\ge1\\2rs=n}}(-1)^{r+s} 
 - 2 \sum_{\substack{s,t\ge1\\4st=n}}(-1)^{s+t}- 4\sum_{\substack{r,s,t\ge1\\2rs+2rt+4st=n}}(-1)^{r+s+t}, \quad &n = 2m,
 \\
 \label{DKM_odd}
 a(n) &= -2\sum_{\substack{s,t\ge 1\\(2s-1)(2t-1)=n}}(-1)^{s + t}-4\sum_{\substack{r,s,t\ge 1\\ 2r(s+t-1)+(2s-1)(2t-1)=n}}(-1)^{r + s + t},\quad &n = 2m + 1.
\end{align}
\end{corollary}
As a result of Theorem \ref{theorem:result-DKM}, we are also able to prove

\begin{theorem} \label{theorem:newAC-identity} For positive integers $n$, we have
\begin{align*}
a(n)&=-4H(4n), & \textup{for \ }n\equiv 1,\ 2, \ 5, \ \textup{6} \pmod 8,\\
a(n)&=6H(4n)=24H(n), & \textup{for \ }n\equiv 3 \pmod 8,\\
a(n)&=0, & \textup{for \ }n\equiv 7 \pmod 8,\\
a(n)&=r_3(n)=r_3(n/4), & \textup{for \ }n\equiv 0 \pmod 4.
\end{align*}
\end{theorem}

\begin{remark}
It is straightforward to show that $|a(n)|$ counts the number of representations of $n$ as $n=x^2+2y^2+2z^2$.
\end{remark}
In Sections \ref{section:prelim-theta} and \ref{section:prelim-Hurwitz} we review theta functions, Appell functions, and Hurwitz class numbers.  In Section \ref{section:newAC} we prove Theorem \ref{theorem:result-DKM}.  In Section \ref{section:FourierCoefficient} we prove intermediate results for the Fourier coefficients $a(n)$, and in Section \ref{section:newAC-FC} we prove Theorem \ref{theorem:newAC-identity}.  Concluding remarks are found in Section \ref{section:conclusion}.

\section*{Acknowledgements}

This work was supported by the Theoretical Physics and Mathematics Advancement Foundation BASIS, agreement No. 20-7-1-25-1.

\section{Preliminaries: theta functions and Appell functions}\label{section:prelim-theta}
We review the basic facts and terminology for theta functions and Appell functions.  We recall the theta function:
\begin{equation*}
j(x;q):=(x)_{\infty}(q/x)_{\infty}(q)_{\infty}.
\end{equation*}
We let $a$ and $m$ are integers with $m$ positive.  We also define
\begin{gather*}
J_{a,m}:=j(q^a;q^m), \ \ J_m:=J_{m,3m}=\prod_{i\ge 1}(1-q^{mi}), \ {\text{and }}\overline{J}_{a,m}:=j(-q^a;q^m).
\end{gather*}

We will frequently use the following identities without mention.  They easily follow from the definitions.  We have
{\allowdisplaybreaks \begin{subequations}
\begin{gather}
\overline{J}_{0,1}=2\overline{J}_{1,4}=\frac{2J_2^2}{J_1},  \ 
\overline{J}_{1,2}=\frac{J_2^5}{J_1^2J_4^2}, \ 
  J_{1,2}=\frac{J_1^2}{J_2}, \ J_{1,4}=\frac{J_1J_4}{J_2}.\notag
\end{gather}
\end{subequations}}%
Also following from the definitions are the general identities:
{\allowdisplaybreaks \begin{subequations}
\begin{gather}
j(q^n x;q)=(-1)^nq^{-\binom{n}{2}}x^{-n}j(x;q), \ \ n\in\mathbb{Z},\label{equation:j-elliptic}\\
j(x;q)=j(q/x;q)\label{equation:j-flip},\\
j(z;q)=\sum_{k=0}^{m-1}(-1)^k q^{\binom{k}{2}}z^k
j\big ((-1)^{m+1}q^{\binom{m}{2}+mk}z^m;q^{m^2}\big ).\label{equation:j-split}
\end{gather}
\end{subequations}}%

\begin{lemma}\label{lemma:1} We have
\begin{equation}
    j(iq;q^2)=\frac{(q^4;q^4)_{\infty}^2}{(q^8;q^8)_{\infty}}.\label{equation:jID-1}
\end{equation}
\end{lemma}

\begin{proof}[Proof of Lemma \ref{lemma:1}] A straightforward product rearrangement yields
\begin{equation*}
j(iq;q^2)=(q^2;q^2)_{\infty}(iq;q^2)_{\infty}(-iq;q^2)_{\infty}
=\frac{(q^4;q^4)_{\infty}^2}{(q^8;q^8)_{\infty}}.\qedhere
\end{equation*}
\end{proof}

We recall the following definition of an Appell function \cite{HM}:
\begin{equation}
m(x,z;q):=\frac{1}{j(z;q)}\sum_{r=-\infty}^{\infty}\frac{(-1)^rq^{\binom{r}{2}}z^r}{1-q^{r-1}xz}.\label{equation:mdef-eq}
\end{equation}
Appell functions are the building blocks of Ramanujan's classical mock theta functions \cite[Section 5]{HM}.  We will use the following properties.  We recall \cite[Corollary 3.2]{HM} 
\begin{align}
m(q,-1;q^2)&=\frac{1}{2}\label{equation:HM-eq3.3},\\
m(-1,q;q^2)&=0,\label{equation:HM-eq3.4}
\end{align}
as well as \cite[Theorem 3.3]{HM}:
\begin{equation}
m(x,z_1;q)-m(x,z_0;q)=\frac{z_0J_1^3j(z_1/z_0;q)j(xz_1z_0;q)}{j(z_0;q)j(z_1;q)j(xz_0;q)j(xz_1;q)}.\label{equation:changing-z}
\end{equation}

\section{Preliminaries:  binary quadratic forms and class numbers}\label{section:prelim-Hurwitz}

We review the basic terminology and facts for binary quadratic forms and class numbers \cite{Co, Za}.  Let
\begin{equation*}
f(x,y):=ax^2+bxy+cy^2
\end{equation*}
be a binary quadratic form.  The discriminant of $f(x,y)$ is denoted $D(f):=b^2-4ac$.

\begin{definition} \cite[Definition $5.3.2$]{Co}
 A positive definite quadratic form $(a,b,c)$ of discriminant $D$ is said to be {\em reduced} if $|b|\le a\le c$ and if, in addition, when one of the two inequalities is an equality (this means that either $|b|=a$ or $a=c$), then $b\ge0$.
\end{definition}

It turns out that the number of reduced forms of a given discriminant $D$ is finite in number \cite[p. 59]{Za}.

\begin{definition} 
A reduced form is defined to be {\em primitive} if $\gcd(a,b,c)=1$, and it is defined to be {\em imprimitive} otherwise. 
\end{definition}

\begin{definition}\cite[p. 226]{Co}
The {\em class number} $h(D)$ is defined to be the number of primitive positive definite reduced quadratic forms of discriminant $D$.
\end{definition}

\begin{definition}\label{definition:Cohen-definition536}  \cite[Definition $5.3.6$]{Co} Let $N$ be a non-negative integer.  The {\em Hurwitz class number} $H(N)$ is defined as follows.
\begin{itemize}
\item[(1)]  If $N\equiv 1, 2 \pmod 4$ then $H(N)=0$.
\item[(2)]  If $N=0$ then $H(0)=-1/12$.
\item[(3)]  For all other cases (i.e. if $N\equiv 0 \ \textup{or } 3 \pmod 4$ and $N>0$) we define $H(N)$ as the class number of not necessarily primitive (positive definite) quadratic forms of discriminant $-N$, except that forms equivalent to $a(x^2+y^2)$ should be counted with weight $1/2$, and those equivalent to $a(x^2+xy+y^2)$ with weight $1/3$.
\end{itemize}
\end{definition}

\begin{lemma} \cite[Lemma $3.5$]{Mo2017B}\label{lemma:mult-4} For positive numbers $n$, where $n\equiv 3 \pmod 4$, we have
{\allowdisplaybreaks \begin{align}
H(4n)=4H(n) \ \textup{for }n\equiv 3\pmod 8,\label{equation:delta-3mod8}\\
H(4n)=2H(n) \ \textup{for }n\equiv 7\pmod 8.\label{equation:delta-7mod8}
\end{align}}%
\end{lemma}

\section{Proof of Theorem \ref{theorem:result-DKM}}\label{section:newAC}
\subsection{The right-hand side} We specialize Theorem \ref{theorem:result}, with $y=ix$, $z=-ix$, $x\to -1$.  We first focus on the right-hand side of (\ref{equation:thm-result}) and set $y=ix$, $z=-ix$.  We obtain
 {\allowdisplaybreaks \begin{align*}
RHS
&=  \frac{(x^2,q^2/x^2;q^2)_{\infty}}{(ix,-ix,-iq/x,iq/x;q)_{\infty}}
\frac{(q;q)_{\infty}^2}{(q^2;q^2)_{\infty}}
 \sum_{k\in \mathbb{Z}}\frac{(-1)^kq^{k^2}(x^2)^k}{1+q^{2k}x} \\
&\ \ \ \ \ + \frac{(ix^2,-iq^2/x^2;q^2)_{\infty}}{(x,ix,q/x,-iq/x;q)_{\infty}}
 \frac{(q;q)_{\infty}^2}{(q^2;q^2)_{\infty}}
 \sum_{k\in \mathbb{Z}} \frac{(-1)^kq^{k^2}(ix^2)^k}{1-q^{2k}ix} \\
&\ \ \ \ \ -2  \frac{(q^2;q^2)_{\infty}^3}{(x,ix,-ix,q/x,-iq/x,iq/x;q)_{\infty}}
 \frac{(ix^2,-ix^2,x^2,-iq^2/x^2,iq^2/x^2,q^2/x^2;q^2)_{\infty}}
 {(-x,-ix,ix,-q^2/x,iq^2/x,-iq^2/x;q^2)_{\infty}} \\
&\ \ \ \ \ + \frac{(-ix^2,iq^2/x^2;q^2)_{\infty}}{(x,-ix,q/x,iq/x;q)_{\infty}}
 \frac{(q;q)_{\infty}^2}{(q^2;q^2)_{\infty}}  
 \sum_{k\in \mathbb{Z}} \frac{(-1)^kq^{k^2}(-ix^2)^k}{1+q^{2k}ix}.
 \end{align*}}%
 
 We want to let $x\to -1$.  We denote the four summands respectively as $S1$, $S2$, $T3$, $S4$.  The only potential singularities are in $S1$ and $T3$.  Let us simplify in order to remove singularities.  This gives
 {\allowdisplaybreaks \begin{align*}
RHS &=  \frac{(1-x^2)(q^2x^2,q^2/x^2;q^2)_{\infty}}{(ix,-ix,-iq/x,iq/x;q)_{\infty}}
\frac{(q;q)_{\infty}^2}{(q^2;q^2)_{\infty}}
 \sum_{k\in \mathbb{Z}}\frac{(-1)^kq^{k^2}(x^2)^k}{1+q^{2k}x} \\
&\qquad + \frac{(ix^2,-iq^2/x^2;q^2)_{\infty}}{(x,ix,q/x,-iq/x;q)_{\infty}}
 \frac{(q;q)_{\infty}^2}{(q^2;q^2)_{\infty}}
 \sum_{k\in \mathbb{Z}} \frac{(-1)^kq^{k^2}(ix^2)^k}{1-q^{2k}ix} \\
&\qquad -2  \frac{(q^2;q^2)_{\infty}^3}{(1-x)(xq,ix,-ix,q/x,-iq/x,iq/x;q)_{\infty}}\\
&\qquad \qquad \cdot \frac{(1-x^2)(ix^2,-ix^2,x^2q^2,-iq^2/x^2,iq^2/x^2,q^2/x^2;q^2)_{\infty}}
 {(1+x)(-xq^2,-ix,ix,-q^2/x,iq^2/x,-iq^2/x;q^2)_{\infty}} \\
&\qquad + \frac{(-ix^2,iq^2/x^2;q^2)_{\infty}}{(x,-ix,q/x,iq/x;q)_{\infty}}
 \frac{(q;q)_{\infty}^2}{(q^2;q^2)_{\infty}}  
 \sum_{k\in \mathbb{Z}} \frac{(-1)^kq^{k^2}(-ix^2)^k}{1+q^{2k}ix}.
 \end{align*}}%
 Simplifying and rearranging factors brings us to
 {\allowdisplaybreaks \begin{align*}
 RHS
 &=  \frac{(1-x)(q^2x^2,q^2/x^2;q^2)_{\infty}}{(ix,-ix,-iq/x,iq/x;q)_{\infty}}
\frac{(q;q)_{\infty}^2}{(q^2;q^2)_{\infty}}
(1+x) \sum_{k\in \mathbb{Z}}\frac{(-1)^kq^{k^2}(x^2)^k}{1+q^{2k}x} \\
&\qquad + \frac{(ix^2,-iq^2/x^2;q^2)_{\infty}}{(x,ix,q/x,-iq/x;q)_{\infty}}
 \frac{(q;q)_{\infty}^2}{(q^2;q^2)_{\infty}}
 \sum_{k\in \mathbb{Z}} \frac{(-1)^kq^{k^2}(ix^2)^k}{1-q^{2k}ix} \\
&\qquad -2  \frac{(q^2;q^2)_{\infty}^3}{(xq,ix,-ix,q/x,-iq/x,iq/x;q)_{\infty}}\\
&\qquad \qquad \cdot \frac{(ix^2,-ix^2,x^2q^2,-iq^2/x^2,iq^2/x^2,q^2/x^2;q^2)_{\infty}}
 {(-xq^2,-ix,ix,-q^2/x,iq^2/x,-iq^2/x;q^2)_{\infty}} \\
&\qquad + \frac{(-ix^2,iq^2/x^2;q^2)_{\infty}}{(x,-ix,q/x,iq/x;q)_{\infty}}
 \frac{(q;q)_{\infty}^2}{(q^2;q^2)_{\infty}}  
 \sum_{k\in \mathbb{Z}} \frac{(-1)^kq^{k^2}(-ix^2)^k}{1+q^{2k}ix}.
\end{align*}}%
Letting $x\to -1$ results in
{\allowdisplaybreaks \begin{align*}
 RHS
 &=  2\frac{(q^2,q^2;q^2)_{\infty}}{(-i,i,iq,-iq;q)_{\infty}}
\frac{(q;q)_{\infty}^2}{(q^2;q^2)_{\infty}}\\
&\qquad + \frac{(i,-iq^2;q^2)_{\infty}}{(-1,-i,-q,iq;q)_{\infty}}
 \frac{(q;q)_{\infty}^2}{(q^2;q^2)_{\infty}}
 \sum_{k\in \mathbb{Z}} \frac{(-1)^kq^{k^2}(i)^k}{1+q^{2k}i} \\
&\qquad -2  \frac{(q^2;q^2)_{\infty}^3}{(-q,-i,i,-q,iq,-iq;q)_{\infty}}
 \cdot \frac{(i,-i,q^2,-iq^2,iq^2,q^2;q^2)_{\infty}}
 {(q^2,i,-i,q^2,-iq^2,iq^2;q^2)_{\infty}} \\
&\qquad + \frac{(-i,iq^2;q^2)_{\infty}}{(-1,i,-q,-iq;q)_{\infty}}
 \frac{(q;q)_{\infty}^2}{(q^2;q^2)_{\infty}}  
 \sum_{k\in \mathbb{Z}} \frac{(-1)^kq^{k^2}(-i)^k}{1-q^{2k}i}.
\end{align*}}%
Now we consider the four summands individually.  We first look at $S1$.  We have
{\allowdisplaybreaks \begin{align}
S1
&=2\frac{(q^2,q^2;q^2)_{\infty}}{(-i,i,iq,-iq;q)_{\infty}}
\frac{(q;q)_{\infty}^2}{(q^2;q^2)_{\infty}}\notag\\
&=\frac{2}{(1-i)(1+i)}\frac{(q^2;q^2)_{\infty}(q;q)_{\infty}^2}{(-iq,iq,iq,-iq;q)_{\infty}}\notag\\
&=\frac{(q^2;q^2)_{\infty}(q;q)_{\infty}^2}{(-q^2;q^2)_{\infty}^2}\notag\\
&=\frac{(q^2;q^2)_{\infty}^3(q;q)_{\infty}^2}{(q^4;q^4)_{\infty}^2}.
\label{equation:S1-final}
\end{align}}%
Now we consider $T3$.  We have
{\allowdisplaybreaks \begin{align}
T3
&=-2  \frac{(q^2;q^2)_{\infty}^3}{(-q,-i,i,-q,iq,-iq;q)_{\infty}}
 \cdot \frac{(i,-i,q^2,-iq^2,iq^2,q^2;q^2)_{\infty}}
 {(q^2,i,-i,q^2,-iq^2,iq^2;q^2)_{\infty}}\notag \\
 &=-\frac{2}{(1+i)(1-i)}  \frac{(q^2;q^2)_{\infty}^3}{(-q,-iq,iq,-q,iq,-iq;q)_{\infty}}\notag\\
&\qquad  \cdot \frac{(1-i)(1+i)}{(1-i)(1+i)}\frac{(iq^2,-iq^2,q^2,-iq^2,iq^2,q^2;q^2)_{\infty}}
 {(q^2,iq^2,-iq^2,q^2,-iq^2,iq^2;q^2)_{\infty}} \notag\\
 &=- \frac{(q^2;q^2)_{\infty}^3}{(-q,-iq,iq;q)_{\infty}^2}\notag\\
 &=- \frac{(q^2;q^2)_{\infty}^3(q;q)_{\infty}^2}{(q^4;q^4)_{\infty}^2}.
 \label{equation:T3-final}
\end{align}}%
Moving on to $S2$ we find that
{\allowdisplaybreaks \begin{align*}
S2
&= \frac{(i,-iq^2;q^2)_{\infty}}{(-1,-i,-q,iq;q)_{\infty}}
 \frac{(q;q)_{\infty}^2}{(q^2;q^2)_{\infty}}
 \sum_{k\in \mathbb{Z}} \frac{(-1)^kq^{k^2}(i)^k}{1+q^{2k}i} \\
 &= \frac{(1-i)(iq^2,-iq^2;q^2)_{\infty}}{2(1+i)(-q,-iq,-q,iq;q)_{\infty}}
 \frac{(q;q)_{\infty}^2}{(q^2;q^2)_{\infty}}
 \sum_{k\in \mathbb{Z}} \frac{(-1)^kq^{k^2}(i)^k}{1+q^{2k}i} \\
 &= \frac{(1-i)(-q^4;q^4)_{\infty}}{2(1+i)(-q^2;q^2)_{\infty}(-q;q)_{\infty}^2}
 \frac{(q;q)_{\infty}^2}{(q^2;q^2)_{\infty}}
 \sum_{k\in \mathbb{Z}} \frac{(-1)^kq^{k^2}(i)^k}{1+q^{2k}i} \\
 &=-\frac{i}{2} \frac{(q^8;q^8)_{\infty}}{(q^4;q^4)_{\infty}^2}
 \frac{(q;q)_{\infty}^4}{(q^2;q^2)_{\infty}^2}
 \sum_{k\in \mathbb{Z}} \frac{(-1)^kq^{k^2}(i)^k}{1+q^{2k}i}.
\end{align*}}%
We use Appell function properties to rewrite $S2$.  This produces
{\allowdisplaybreaks \begin{align*}
S2
 &=-\frac{i}{2} \frac{(q^8;q^8)_{\infty}}{(q^4;q^4)_{\infty}^2}
 \frac{(q;q)_{\infty}^4}{(q^2;q^2)_{\infty}^2}
 \sum_{k\in \mathbb{Z}} \frac{(-1)^kq^{k^2}(i)^k}{1+q^{2k}i} \\
  &=-\frac{i}{2} \frac{(q^8;q^8)_{\infty}}{(q^4;q^4)_{\infty}^2}
 \frac{(q;q)_{\infty}^4}{(q^2;q^2)_{\infty}^2}
j(iq;q^2)m( -q,iq;q^2).
\end{align*}}%
Using (\ref{equation:changing-z}) with $q\to q^2$, $x\to -q$, $z_1\to iq$, $z_0\to -1$, as well as (\ref{equation:HM-eq3.3}), we can write
{\allowdisplaybreaks \begin{align*}
S2
&=-\frac{i}{2} \frac{(q^8;q^8)_{\infty}}{(q^4;q^4)_{\infty}^2}
 \frac{(q;q)_{\infty}^4}{(q^2;q^2)_{\infty}^2}
j(iq;q^2)m( -q,iq;q^2)\\
&=-\frac{i}{2} \frac{(q^8;q^8)_{\infty}}{(q^4;q^4)_{\infty}^2}
 \frac{(q;q)_{\infty}^4}{(q^2;q^2)_{\infty}^2}
j(iq;q^2)\left ( \frac{1}{2} -\frac{J_2^3j(-iq;q^2)j(iq^2;q^2)}{j(-1;q^2)j(iq;q^2)j(q;q^2)j(-iq^2;q^2)}\right ).
\end{align*}}%
Using (\ref{equation:j-flip}) allows us to write
\begin{align*}
S2
&=-\frac{i}{2} \frac{(q^8;q^8)_{\infty}}{(q^4;q^4)_{\infty}^2}
 \frac{(q;q)_{\infty}^4}{(q^2;q^2)_{\infty}^2}
j(iq;q^2)\left ( \frac{1}{2} -\frac{J_2^3j(iq;q^2)j(iq^2;q^2)}{j(-1;q^2)j(iq;q^2)j(q;q^2)j(i;q^2)}\right ).
\end{align*}
Simplifying, using (\ref{equation:jID-1}), and employing (\ref{equation:j-elliptic}) brings us to
{\allowdisplaybreaks \begin{align*}
S2
&=-\frac{i}{2} \frac{(q^8;q^8)_{\infty}}{(q^4;q^4)_{\infty}^2}
 \frac{(q;q)_{\infty}^4}{(q^2;q^2)_{\infty}^2}
\frac{(q^4;q^4)_{\infty}^2}{(q^8;q^8)_{\infty}}\left ( \frac{1}{2} -i\frac{J_2^3j(i;q^2)}{j(-1;q^2)j(q;q^2)j(i;q^2)}\right )\\
&=-\frac{i}{2} \frac{(q^8;q^8)_{\infty}}{(q^4;q^4)_{\infty}^2}
 \frac{(q;q)_{\infty}^4}{(q^2;q^2)_{\infty}^2}
\frac{(q^4;q^4)_{\infty}^2}{(q^8;q^8)_{\infty}}\left ( \frac{1}{2} 
-i\frac{(q^2;q^2)_{\infty}^3(q^2;q^2)_{\infty}(q^2;q^2)_{\infty}}{2(q^4;q^4)_{\infty}^2(q;q)_{\infty}^2}\right )\\
&=-\frac{i}{4} 
 \frac{(q;q)_{\infty}^4}{(q^2;q^2)_{\infty}^2}
 -\frac{1}{4} 
 \frac{(q;q)_{\infty}^2(q^2;q^2)_{\infty}^3}{(q^4;q^4)_{\infty}^2}.
\end{align*}}%
The sum $S4$ is the complex conjugate of $S2$.  Hence
\begin{equation}
S2+S4
=-\frac{1}{2} 
 \frac{(q;q)_{\infty}^2(q^2;q^2)_{\infty}^3}{(q^4;q^4)_{\infty}^2}
=-\frac{1}{2}\left ( \frac{(q;q)_{\infty}}{(-q;q)_{\infty}}\right ) 
\left( \frac{(q^2;q^2)_{\infty}}{(-q^2;q^2)_{\infty}}\right )^2.
\label{equation:S2S4-final}
\end{equation}
Recalling (\ref{equation:S1-final}), (\ref{equation:T3-final}), and (\ref{equation:S2S4-final}) yields
{\allowdisplaybreaks \begin{align}
RHS &= S1 + T3 + S2 + S4 \notag \\
&=\frac{(q^2;q^2)_{\infty}^3(q;q)_{\infty}^2}{(q^4;q^4)_{\infty}^2}
- \frac{(q^2;q^2)_{\infty}^3(q;q)_{\infty}^2}{(q^4;q^4)_{\infty}^2}
-\frac{1}{2}\left ( \frac{(q;q)_{\infty}}{(-q;q)_{\infty}}\right ) 
\left( \frac{(q^2;q^2)_{\infty}}{(-q^2;q^2)_{\infty}}\right )^2\notag\\
&=-\frac{1}{2}\left ( \frac{(q;q)_{\infty}}{(-q;q)_{\infty}}\right ) 
\left( \frac{(q^2;q^2)_{\infty}}{(-q^2;q^2)_{\infty}}\right )^2.
\label{equation:RHS-final}
\end{align}}%

\subsection{The left-hand side}
We first rewrite the left-hand side in the following form.  Using the left-hand side as in Theorem \ref{theorem:result} gives
{\allowdisplaybreaks \begin{align*}
&\left ( \sum_{r,s,t\ge 0}+\sum_{r,s,t<0} \right )q^{rs+rt+st}x^{r}y^{s}z^{t}\\
&\qquad =1+
 \sum_{r,s\ge 1}q^{rs}x^{r}y^{s}
 + \sum_{r,t\ge 1}q^{rt}x^{r}z^{t}
 + \sum_{s,t\ge 1}q^{st}y^{s}z^{t}\\
 & \qquad \qquad + \sum_{r\ge 1}x^{r}
 + \sum_{s\ge 1}y^{s}
 + \sum_{t\ge 1}z^{t}
 +\left ( \sum_{r,s,t> 0}+\sum_{r,s,t<0} \right )q^{rs+rt+st}x^{r}y^{s}z^{t}.
\end{align*}}%
We make the substitutions $(y,z)\to(ix,-ix)$, to obtain
{\allowdisplaybreaks \begin{align*}
&\left ( \sum_{r,s,t\ge 0}+\sum_{r,s,t<0} \right )q^{rs+rt+st}x^{r}(ix)^{s}(-ix)^{t}\\
&\qquad =1+
 \sum_{r,s\ge 1}q^{rs}x^{r+s}i^{s}
 + \sum_{r,t\ge 1}q^{rt}x^{r+t}(-i)^{t}
 + \sum_{s,t\ge 1}q^{st}x^{s+t}(-1)^{t}i^{s+t}\\
 & \qquad \qquad + \sum_{r\ge 1}x^{r}
 + \sum_{s\ge 1}(ix)^{s}
 + \sum_{t\ge 1}(-ix)^{t}
 +\left ( \sum_{r,s,t> 0}+\sum_{r,s,t<0} \right )q^{rs+rt+st}x^{r+s+t}(-1)^{t}i^{s+t}\\
 &\qquad =1+
 \sum_{r,s\ge 1}q^{rs}x^{r+s}i^{s}
 + \sum_{r,t\ge 1}q^{rt}x^{r+t}(-i)^{t}
 + \sum_{s,t\ge 1}q^{st}x^{s+t}(-1)^{t}i^{s+t}\\
 & \qquad \qquad + \frac{x}{1-x} + \frac{ix}{1-ix} - \frac{ix}{1+ix}
 +\left ( \sum_{r,s,t> 0}+\sum_{r,s,t<0} \right )q^{rs+rt+st}x^{r+s+t}(-1)^{t}i^{s+t}.
\end{align*}}%
We now let $x\to -1$ to have
{\allowdisplaybreaks \begin{align}
LHS 
  & =- \frac{1}{2} +
 \sum_{r,s\ge 1}q^{rs}(-1)^{r+s}i^{s}
 + \sum_{r,t\ge 1}q^{rt}(-1)^{r}i^{t}
 + \sum_{s,t\ge 1}q^{st}(-1)^{s}i^{s+t}\notag \\
 & \qquad \qquad 
 +\left ( \sum_{r,s,t> 0}+\sum_{r,s,t<0} \right )q^{rs+rt+st}(-1)^{r+s}i^{s+t}.
 \label{equation:LHS-imaginary}
\end{align}}%
We now rewrite the above expression without imaginary terms.  We focus on the top line of (\ref{equation:LHS-imaginary}) and start expanding to have
{\allowdisplaybreaks \begin{align*}
LHS 
& =-\frac{1}{2} +
 \sum_{r,s\ge 1}q^{2rs}(-1)^{r+s}+ \sum_{r,s\ge 1}q^{r(2s-1)}(-1)^{r-1}i^{2s-1}\\
  &\qquad + \sum_{r,t\ge 1}q^{2rt}(-1)^{r}i^{2t}+ \sum_{r,t\ge 1}q^{r(2t-1)}(-1)^{r}i^{2t-1}\\
  &\qquad + \sum_{s,t\ge 1}q^{4st}(-1)^{2s}i^{2s+2t}+ \sum_{s,t\ge 1}q^{(2s-1)(2t-1)}(-1)^{2s-1}i^{2s+2t-2}\\
  &\qquad + \sum_{s,t\ge 1}q^{2s(2t-1)}(-1)^{2s}i^{2s+2t-1}+ \sum_{s,t\ge 1}q^{(2s-1)2t}(-1)^{2s-1}i^{2s-1+2t}\\
 & \qquad \qquad 
 +\left ( \sum_{r,s,t> 0}+\sum_{r,s,t<0} \right )q^{rs+rt+st}(-1)^{r+s}i^{s+t}\\
 & =-\frac{1}{2} +
 \sum_{r,s\ge 1}q^{2rs}(-1)^{r+s}+  i\sum_{r,s\ge 1}q^{r(2s-1)}(-1)^{r+s}\\
  &\qquad + \sum_{r,t\ge 1}q^{2rt}(-1)^{r+t}-i\sum_{r,t\ge 1}q^{r(2t-1)}(-1)^{r+t}\\
  &\qquad + \sum_{s,t\ge 1}q^{4st}(-1)^{s+t}+ \sum_{s,t\ge 1}q^{(2s-1)(2t-1)}(-1)^{s+t}\\
  &\qquad + i\sum_{s,t\ge 1}q^{2s(2t-1)}(-1)^{s+t}- i\sum_{s,t\ge 1}q^{(2s-1)2t}(-1)^{s+t}\\
 & \qquad \qquad 
 +\left ( \sum_{r,s,t> 0}+\sum_{r,s,t<0} \right )q^{rs+rt+st}(-1)^{r+s}i^{s+t}\\
 & =-\frac{1}{2} +
2 \sum_{r,s\ge 1}q^{2rs}(-1)^{r+s}
 + \sum_{s,t\ge 1}q^{4st}(-1)^{s+t}+ \sum_{s,t\ge 1}q^{(2s-1)(2t-1)}(-1)^{s+t}\\
 & \qquad \qquad 
 +\left ( \sum_{r,s,t> 0}+\sum_{r,s,t<0} \right )q^{rs+rt+st}(-1)^{r+s}i^{s+t}.
\end{align*}}%
Now we focus on the bottom line of (\ref{equation:LHS-imaginary}).  We have
{\allowdisplaybreaks \begin{align*}
LHS 
& =- \frac{1}{2} +
2 \sum_{r,s\ge 1}q^{2rs}(-1)^{r+s}
 + \sum_{s,t\ge 1}q^{4st}(-1)^{s+t}+ \sum_{s,t\ge 1}q^{(2s-1)(2t-1)}(-1)^{s+t}\\
 & \qquad 
 + \sum_{r,s,t> 0}q^{r2s+r2t+4st}(-1)^{r+2s}i^{2s+2t}\\
&\qquad  + \sum_{r,s,t> 0}q^{r(2s-1)+r(2t-1)+(2s-1)(2t-1)}(-1)^{r+2s-1}i^{2s+2t-2}\\
 & \qquad + \sum_{r,s,t> 0}q^{r2s+r(2t-1)+2s(2t-1)}(-1)^{r+2s}i^{2s+2t-1}\\
 & \qquad   + \sum_{r,s,t> 0}q^{r(2s-1)+r2t+(2s-1)2t}(-1)^{r+2s-1}i^{2s+2t-1}\\
  & \qquad 
 + \sum_{r,s,t< 0}q^{r2s+r2t+4st}(-1)^{r+2s}i^{2s+2t}\\
&\qquad  + \sum_{r,s,t< 0}q^{r(2s+1)+r(2t+1)+(2s+1)(2t+1)}(-1)^{r+2s+1}i^{2s+2t+2}\\
 & \qquad + \sum_{r,s,t< 0}q^{r2s+r(2t+1)+2s(2t+1)}(-1)^{r+2s}i^{2s+2t+1}\\
 & \qquad   + \sum_{r,s,t< 0}q^{r(2s+1)+r2t+(2s+1)2t}(-1)^{r+2s+1}i^{2s+2t+1}.
 \end{align*}}%
 Continuing some more, we have
{\allowdisplaybreaks \begin{align*}
 LHS
& =- \frac{1}{2} +
2 \sum_{r,s\ge 1}q^{2rs}(-1)^{r+s}
 + \sum_{s,t\ge 1}q^{4st}(-1)^{s+t}+ \sum_{s,t\ge 1}q^{(2s-1)(2t-1)}(-1)^{s+t}\\
 & \qquad 
 + \sum_{r,s,t> 0}q^{2rs+2rt+4st}(-1)^{r+s+t}
   + \sum_{r,s,t> 0}q^{2rs-r+2rt-r+4st-2s-2t+1}(-1)^{r+s+t}\\
 & \qquad -i \sum_{r,s,t> 0}q^{2rs+2rt-r+4st-2s}(-1)^{r+s+t}
    + i\sum_{r,s,t> 0}q^{2rs-r+2rt+4st-2t}(-1)^{r+s+t}\\
  & \qquad 
 + \sum_{r,s,t< 0}q^{r2s+r2t+4st}(-1)^{r+s+t}
   + \sum_{r,s,t< 0}q^{2rs+r+2rt+r+4st+2s+2t+1}(-1)^{r+s+t}\\
 & \qquad +i \sum_{r,s,t< 0}q^{2rs+2rt+r+4st+2s}(-1)^{r+s+t}
    -i \sum_{r,s,t< 0}q^{2rs+r+2rt+4st+2t}(-1)^{r+s+t}\\
 & =- \frac{1}{2} +
2 \sum_{r,s\ge 1}q^{2rs}(-1)^{r+s}
 + \sum_{s,t\ge 1}q^{4st}(-1)^{s+t}+ \sum_{s,t\ge 1}q^{(2s-1)(2t-1)}(-1)^{s+t}\\
 & \qquad 
 + \left (\sum_{r,s,t> 0}+\sum_{r,s,t< 0} \right ) q^{2rs+2rt+4st}(-1)^{r+s+t}
  +2 \sum_{r,s,t> 0}q^{2rs+2rt+4st-2r-2s-2t+1}(-1)^{r+s+t}.
\end{align*}}%
A minor rewrite yields
{\allowdisplaybreaks \begin{align}
 LHS
&=- \frac{1}{2} +
2 \sum_{r,s\ge 1}q^{2rs}(-1)^{r+s}
 + \sum_{s,t\ge 1}q^{4st}(-1)^{s+t}+ \sum_{s,t\ge 1}q^{(2s-1)(2t-1)}(-1)^{s+t}
   \label{equation:LHS-final}\\
 & \qquad 
 + 2\sum_{r,s,t\ge 1} q^{2rs+2rt+4st}(-1)^{r+s+t}
  +2 \sum_{r,s,t\ge 1}q^{r(2s-1)+r(2t-1)+(2s-1)(2t-1)}(-1)^{r+s+t}.
\notag
\end{align}}%

\subsection{Proof of Theorem \ref{theorem:result-DKM}}  Here we simply equation the final form of the right-hand side (\ref{equation:RHS-final}) with the final form of the left-hand side (\ref{equation:LHS-final}).

\section{Evaluating the Fourier coefficients}\label{section:FourierCoefficient}

We recall the divisor function
\begin{equation*}
\sigma_k(n)=\sum_{n\mathop{\raisebox{-2pt}{\vdots}}d}d^{k},
\end{equation*}
where $\sigma_{0}(n)$ counts the number of divisors of $n$. In the following propositions we consider the coefficients in the relation \eqref{relation_DKM} depending on the remainder of $n$ divided by $4$. 
\begin{proposition}\label{proposition:prop-4mPlus2} For $n=4m+2$, we have
\begin{align*}
a(n)=-4\sigma_0\left ( \frac{n}{2}\right )-4\sum_{\substack{r,s,t\ge 1\\(2s+r)(2t+r)=n+r^2}}1.
\end{align*}
\end{proposition}

\begin{proposition}\label{proposition:prop-4mPlus1} For $n=4m+1$, we have
\begin{align*}
a(n)=-2\sigma_0(n)-4\sum_{\substack{r,s,t\ge 1\\ (2s-1+r)(2t-1+r)=n+r^2}}1.
\end{align*}
\end{proposition}

\begin{proposition}\label{proposition:prop-4mPlus3} For $n=8m+3$, we have
\begin{align*}
a(n)=2\sigma_0(n)+4\sum_{\substack{r,s,t\ge 1\\ (2s-1+r)(2t-1+r)=n+r^2}}1.
\end{align*}
\end{proposition}

\begin{proposition}\label{proposition:prop-8mPlus7}  For $n=8m+7$, we have
\begin{align*}
a(n)=2\sigma_0(n)-4\sum_{\substack{r,s,t\ge 1\\(2s-1+r)(2t-1+r)=n+r^2}}(-1)^{r+s+t} = 0.
\end{align*}
\end{proposition}

\begin{proposition}\label{proposition:prop-4m}  For $n=4m$, we have
\begin{align*}
a(n)=r_3(n)=r_3(n/4).
\end{align*}
\end{proposition}

\begin{proof}[Proof of Proposition \ref{proposition:prop-4mPlus2}]
Equation \eqref{DKM_even} for $n = 4m + 2$ gives
\begin{align*}
a(n) = -4\sum_{\substack{r,s\ge1\\2rs=n}}(-1)^{r+s}-4\sum_{\substack{r,s,t\ge1\\2rs+2rt+4st=n}}(-1)^{r+s+t}.
\end{align*}
If $2rs = n = 4m + 2$ then $rs = 2m + 1$ and both $r$ and $s$ are odd. Hence $r + s$ is even and every term in the first sum equals $1$. Similarly, if $2rs + 2rt + 4st = n = 4m + 2$ then $rs + rt + 2st = 2m + 1$. Hence $r(s + t)$ is odd, $r + s + t$ is even, and every term in the second sum also equals $1$. Therefore
\begin{gather*}
a(n)
=-4\sum_{\substack{r,s\ge 1 \\ rs=\frac{n}{2}}}1-4\sum_{\substack{r,s,t\ge 1\\2rs+2rt+4st=n}}1=-4\sigma_0\left ( \frac{n}{2}\right )-4\sum_{\substack{r,s,t\ge 1\\(2s+r)(2t+r)=n+r^2}}1.\qedhere
\end{gather*}
\end{proof}

\begin{proof}[Proof of Proposition \ref{proposition:prop-4mPlus1}]
Express $a(n)$ by equation \eqref{DKM_odd}. If $(2s - 1)(2t - 1) = n = 4m + 1$ then $4st - 2(s + t) = n - 1= 4m$ and $s + t$ is even. Hence every term in the first sum of  \eqref{DKM_odd} is $1$. The equality $2r(s+t-1)+(2s-1)(2t-1)=n$ rewrites as $4st + 2(r - 1)(s + t - 1) =n + 1 = 4m + 2$. Hence $(r - 1)(s + t- 1)$ is odd, $r$ and $s + t$ are even and every term in the second sum of  \eqref{DKM_odd} is also $1$. Therefore 
\begin{gather*}
    a(n) = -2\sum_{\substack{s,t\ge 1\\(2s-1)(2t-1)=n}}1-4\sum_{\substack{r,s,t\ge 1\\ 2r(s+t-1)+(2s-1)(2t-1)=n}}1  =-2\sigma_0(n)-4\sum_{\substack{r,s,t\ge 1\\ (2s-1+r)(2t-1+r)=n+r^2}}1.
\end{gather*}

\end{proof}

\begin{proof}[Proof of Proposition \ref{proposition:prop-4mPlus3}]
We proceed similarly to the proof of Proposition \ref{proposition:prop-4mPlus1}. If $(2s - 1)(2t - 1) = n = 8m + 3$ then $4st - 2(s + t) = 8m + 2$ and $s + t$ is odd. Hence every term in the first sum of  \eqref{DKM_odd} is $-1$. Second equation becomes $4st + 2(r - 1)(s + t - 1) =n + 1 = 8m + 4$. If $s + t - 1$ is odd then $r - 1$ has to be even and $r + s + t$ is odd. If $s + t - 1$ is even then $st$ is even and $2(r - 1)(s + t - 1)\equiv 4\pmod 8$. Hence $r$ is even and $r +s + t$ is again odd. Thus $r + s + t$ is always odd and every term in the second sum is $-1$. Therefore
\begin{gather*}
    a(n) = 2\sum_{\substack{s,t\ge 1\\(2s-1)(2t-1)=n}}1+4\sum_{\substack{r,s,t\ge 1\\ 2r(s+t-1)+(2s-1)(2t-1)=n}}1  =2\sigma_0(n)+4\sum_{\substack{r,s,t\ge 1\\ (2s-1+r)(2t-1+r)=n+r^2}}1.\qedhere
\end{gather*}

\end{proof}

\begin{proof}[Proof of Proposition \ref{proposition:prop-8mPlus7}]
Similarly to the proof of Proposition \ref{proposition:prop-4mPlus3} we use relation \eqref{DKM_odd} and notice that each term in the first sum is $-1$. We get
\begin{gather*}
a(n)=2\sigma_0(n)-4\sum_{\substack{r,s,t\ge 1\\(2s-1+r)(2t-1+r)=n+r^2}}(-1)^{r+s+t}.
\end{gather*}
It can be shown that $a(n)=0$ in this case (see Case 4 of the proof of Theorem \ref{theorem:theorem1-nPlusr2}).
\end{proof}

 \begin{proof}[Proof of Proposition \ref{proposition:prop-4m}] If $n$ is divisible by $4$ and $n$ is represented as a sum of three squares $n = a^2 + b^2 + c^2$ then $a,b$ and $c$ must be even. To each representation of $n$ as a sum of three squares we can match a representation of $n/4$ by simply dividing all elements by $4$; we get $\frac{n}{4} = (\frac{a}{2})^2+ (\frac{b}{2})^2+(\frac{c}{2})^2$. Hence $r_3(n) 
=r_3(n/4)$. To show $r_3(n) = |a(n)|$ let us construct another bijection. For a triple $(a,b,c)$ satisfying $n = a^2 + b^2 + c^2$ we can match the triple $(a, (b+c)/2, (b - c)/2)$ satisfying $n = a^2 + 2\left(\frac{b + c}{2}\right)^2 + 2\left(\frac{b - c}{2}\right)^2$. On the other hand for a triple $(p,q,r)$ such that $n = p^2 + 2q^2 + 2r^2$ we can always match the triple $(p, q + r, q - r)$ satisfying $n = p^2 + (q + r)^2 + (q - r)^2$. These mappings are inverse to each other, hence $r_3(n) = |a(n)|$. 

Let us show that $a(n)$ is non-negative for $n=4m$. By Theorem \ref{theorem:result}, using the notation introduced in Section \ref{section:prelim-theta}, we obtain
\begin{align*}
    {J}_{1,2} {J}^2_{2,4}= j(q; q^2)j(q^2; q^4)^2   = \left ( \frac{(q;q)_{\infty}}{(-q;q)_{\infty}}\right ) 
\left( \frac{(q^2;q^2)_{\infty}}{(-q^2;q^2)_{\infty}}\right )^2=\sum_{n=0}^{\infty}a(n)q^{n}. 
\end{align*}
Notice that relation \eqref{equation:j-split} implies
\begin{align} \label{identityj}
  j(z; q) = j(-qz^2; q^4) -z j(-q^3z^2; q^4).
\end{align}
Making the substitutions $z\mapsto q, q\mapsto q^2$ in \eqref{identityj}, we get
\begin{align*}
    {J}_{1,2} = \overline{J}_{4,8} -q\overline{J}_{0,8}.
\end{align*}
Similarly,
\begin{align*}
  {J}_{2,4} = \overline{J}_{8,16} -q^2\overline{J}_{0,16}.
\end{align*}
Denote $A:=  \overline{J}_{4,8}, \ B:=\overline{J}_{0,8}, 
\ C:=\overline{J}_{8,16}, \ D:=\overline{J}_{0,16}$. We have 
\begin{align*}
   {J}_{1,2} {J}^2_{2,4}= (A-qB)(C-q^2D)^2 = AC^2-qBC^2-2ACDq^2+2BCDq^3+AD^2q^4-AD^2q^5. 
\end{align*}
We see that the coefficients for terms of the form $q^{4m}$ are non-negative, thus $a(n)$ is non-negative. Hence $a(n) = r_3(n)$, the proof is concluded.  
\end{proof}

\section{Proof of Theorem \ref{theorem:newAC-identity}}\label{section:newAC-FC}

Theorem \ref{theorem:newAC-identity} is an immediate consequence of the propositions of Section \ref{section:FourierCoefficient} and the following theorem:

\begin{theorem} \label{theorem:theorem1-nPlusr2} We have
{\allowdisplaybreaks \begin{gather}
\sum_{\substack{r,s,t\ge 1\\(2s+r)(2t+r)=n+r^2}}1
=H(4n)-\sigma_0\left ( \frac{n}{2}\right ),\   \textup{for} \ n\equiv 2 \mod 4,
\\
\sum_{\substack{r,s,t\ge 1\\(2s-1+r)(2t-1+r)=n+r^2}}1
=H(4n)-\frac{1}{2}\sigma_0\left ( n\right ),\   \textup{for} \ n\equiv 1 \mod 4,
\label{equation:theorem1-n1mod4}\\
\sum_{\substack{r,s,t\ge 1\\(2s-1+r)(2t-1+r)=n+r^2}}1
=6H(n)-\frac{1}{2}\sigma_0\left ( n\right ),\   \textup{for} \ n\equiv 3 \mod 8,
\label{equation:theorem1-n3mod8}\\
 \sum_{\substack{r,s,t\ge 1\\(2s-1+r)(2t-1+r)=n+r^2}}(-1)^{r + s + t}
=\frac{1}{2}\sigma_0\left ( n\right ),\  \textup{for} \ n\equiv 7 \mod 8.
\end{gather}}%
\end{theorem}

\begin{proof}[Proof of Theorem \ref{theorem:theorem1-nPlusr2}]

Notice that for any equivalence class there exists one and only one reduced representative. Therefore, we can prove the equations above by establishing mappings from the set of solutions of equations $(2s + r)(2t + r) = n + r^2$ and $(2s  - 1 + r)(2t - 1 + r) = n + r^2$ onto the set of reduced quadratic forms of discriminant $-4n$. 

The main trick which we are going to use to establish the mappings is
\begin{equation}
b^2 - ac = (a - b)^2 - a(a + c - 2b). 
\label{equation:trick}
\end{equation}
The mappings will be constructed slightly differently for different cases of the theorem. The general outline is to divide the set of solution triples into the following six categories and then work with each of them independently:
\begin{enumerate}
    \item $r \leq 2s - 1\cdot\chi(n \centernot\vdots 2) \leq 2t - 1\cdot\chi(n \centernot\vdots 2) $,
    \item $2t - 1\cdot\chi(n \centernot\vdots 2)  \leq r \leq 2s - 1\cdot \chi(n \centernot\vdots 2) $,
    \item $2s - 1\cdot \chi(n \centernot\vdots 2)  \leq 2t - 1\cdot \chi(n \centernot\vdots 2)  \leq r$,
    \item $2s - 1\cdot\chi(n \centernot\vdots 2)  < r < 2t - 1\cdot\chi(n \centernot\vdots 2) $,
    \item $2t - 1\cdot\chi(n \centernot\vdots 2)  < 2s - 1\cdot\chi(n \centernot\vdots 2)  < r$, 
    \item $r < 2t - 1\cdot\chi(n \centernot\vdots 2)  < 2s - 1\cdot\chi(n \centernot\vdots 2) $.
\end{enumerate}
Above, $\chi$ is the indicator function.

It is easy to see that each solution falls into at least one of the categories, and two of the categories can intersect iff there exists a solution triple with $r = 2s - 1\cdot\chi(n \centernot\vdots 2)  = 2t - 1\cdot\chi(n \centernot\vdots 2)$. If this is the case, then categories $(1)$, $(2)$, and $(3)$ all contain this triple but otherwise do not intersect. This, in turn, happens iff $n = 3k^2$ for some $k \in \mathbb Z_+$.

\medskip

\textbf{Case 1:} $n = 4m + 2.$ 

\medskip

A multiple of $x^2 + y^2$ always has a discriminant of the type $-4k^2$. A multiple of $x^2 + xy + y^2$ has a discriminant of the type $-3k^2$. The highest power of $2$ which divides a discriminant of one of these types is even. If $n = 4m + 2$, then $4n = 16m + 8$ which is divisible by $8$ and not by $16$. Therefore, no multiple of $x^2 + y^2$ or $x^2 + xy + y^2$ has discriminant $-4n$, and $H(4n)$ is simply the number of reduced quadratic forms of discriminant $-4n$. Also, note that if $(r, s, t)$ is a solution triple then $r$ is an odd number. Otherwise, since $n = (2s + r)(2t + r) - r^2$, $n$ is divisible by 4.

Let us divide the set of reduced quadratic forms $ax^2 + bxy + cy^2$ of discriminant $-4n$ into seven categories
\begin{enumerate}
    \item Reduced forms which have $b \ \vdots \ 2$, $b > 0$ but $b\centernot\vdots 4$, $a \centernot\vdots 2$, $c \centernot\vdots 2$,
    \item Reduced forms which have $b \ \vdots \ 4$, $b > 0$, $c\ \vdots \ 2$, $a \centernot\vdots 2$,
     \item Reduced forms which have $b \ \vdots \ 4$, $b > 0$, $a\ \vdots \ 2$, $c \centernot\vdots 2$,
     \item Reduced forms which have $b \ \vdots \ 4$, $b < 0$, $c\ \vdots \ 2$, $a \centernot\vdots 2$,
    \item Reduced forms which have $b \ \vdots \ 4$, $b < 0$, $a\ \vdots \ 2$, $c \centernot\vdots 2$,
     \item Reduced forms which have $b \ \vdots \ 2$, $b < 0$ but $b\centernot\vdots 4$, $a \centernot\vdots 2$, $c \centernot\vdots 2$,
      \vspace{0.15cm}
    \item Reduced forms which have $b = 0$.    
\end{enumerate}

First, we should justify this division. From the equation $b^2 - 4ac = -4n = -16m - 8$ it follows that $b \ \vdots \ 2$. In case if $b \ \vdots \ 2$ but $b \centernot\vdots 4$, the product $ac$ should be odd because otherwise $b^2 - 4ac$ is not divisible by $8$. If $b \ \vdots \ 4$, then $ac \ \vdots \ 2$ but $ac \centernot\vdots 4$ because otherwise $b^2 - 4ac$ is either not divisible by $8$ or divisible by $16$.

Recall the categorization of solution triples $(r, s, t)$. We build the mapping separately for each of the categories. We use the trick $\ref{equation:trick}$ and equation $(2s + r)(2t + r) = n + r^2$ to show that these quadratic forms have discriminant $-4n$:
\begin{enumerate}
    \item $a = 2s + r$, $b = 2r$, $c = 2t + r$; \\
    $b^2 - 4ac = 4r^2 - 4(2s + r)(2t + r) = -4n$,
    \item  $a = 2t + r$, $b  = 4t$, $c = 2t + 2s$; \\ $b^2 - (2a)(2c) = (2a - b)^2 - (2a)(2a + 2c - 2b) = 4r^2 - 4(2t + r)(2s + r) = -4n$, 
    \item $a = 2t + 2s$, $b = 4s$, $c = 2s + r$; \\
    $b^2 - (2a)(2c) = (2c - b)^2 - (2c)(2a + 2c - 2b) = 4r^2 - 4(2s + r)(2t + r) = -4n$,
      \item $a = 2s + r$, $b  = -4s$, $c = 2t + 2s$; \\ $b^2 - (-2a)(-2c) = (-2a - b)^2 - (-2a)(-2a - 2c - 2b) = 4r^2 - 4(2s + r)(2t + r) = -4n$, 
    
    \item  $a = 2t + 2s$, $b  = -4t$, $c = 2t + r$; \\ $b^2 - (-2a)(-2c) = (-2c - b)^2 - (-2c)(-2c - 2a - 2b) = 4r^2 - 4(2t + r)(2s + r) = -4n$,
     \item  $a = 2t + r$, $b = -2r$, $c = 2s + r$; \\
    $b^2 - 4ac = 4r^2 - 4(2t + r)(2s + r) = -4n$.

\end{enumerate}
Let us verify that these quadratic forms are indeed reduced. That is, in each case we need to check that $|b| \leq a \leq c$. In addition, if $|b| = a$ or $a = c$ then $b$ has to be non-negative.
\begin{enumerate}
    \item $2r \leq 2s + r \leq 2t + r$ is equivalent to $r \leq 2s \leq 2t$; since $b > 0$, the second condition is satisfied,
    \item $4t \leq 2t + r \leq 2t + 2s$ is equivalent to $2t \leq r \leq 2s$; since $b > 0$, the second condition is satisfied,
    \item $4s \leq 2t + 2s \leq 2s + r$ is equivalent to $2s \leq 2t \leq r$; since $b > 0$, the second condition is satisfied,
     \item $4s < 2s + r < 2t + 2s$ is equivalent to $2s < r < 2t$; since $|b| \neq a$, $a \neq c$, the second condition is satisfied,
    \item $4t < 2t + 2s < 2t + r$ is equivalent to $2t < 2s < r$; since $|b| \neq a$, $a \neq c$, the second condition is satisfied,
    \item $2r < 2t + r < 2s + r$ is equivalent to $r < 2t < 2s$; since $|b| \neq a$, $a \neq c$, the second condition is satisfied.
\end{enumerate}
Notice that the introduced transform maps a solution triple $(r, s, t)$ of category $(i)$ onto a reduced quadratic form of category $(i)$.
Every  quadratic form in category $(i), \ i < 7$ has a preimage and any two solution triples in the same category have distinct images under this transform, since, clearly, for each category there exists an inverse transform. Therefore, this mapping is a bijection onto the set of reduced quadratic forms of categories $(1) - (6)$. 

To prove the claim of the theorem, it only remains to show that the size of category $(7)$ equals $\sigma_0\left(\frac{n}{2}\right)$. Notice that any quadratic form with $b = 0$ corresponds to a pair of numbers $a$ and $c$ such that $a \leq c$ and $ac = n$. Since $n$ is not a square, the number of such pairs equals $\frac{1}{2}\sigma_0(n)$ which, since $n$ is divisible by 2 and not by 4, equals $\sigma_0\left(\frac{n}{2}\right)$.

\medskip

\textbf{Case 2:} $n = 4m + 1.$

\medskip

This case is very similar to Case 1. A multiple of $x^2 + xy + y^2$ cannot have the discriminant of $-4n$. If this was the case, then $n$ would have the form of $3k^2$, however $n \equiv 1 \pmod{4}$ and $3k^2 \not\equiv 1 \pmod{4}$. Also, a multiple of $x^2 + y^2$ has the discriminant of $-4n$ iff $n$ is a square number. By definition of Hurwitz class number, this quadratic form has weight $\frac{1}{2}$. Thus, $H(4n)$ equals the number of reduced quadratic forms of discriminant $-4n$ if $n$ is not a square and it is $\frac{1}{2}$ less than that otherwise. Notice that $r \ \vdots \ 2$ because otherwise LHS of equation $(2s - 1 + r)(2t - 1 + r) = n + r^2$ is divisible by $4$ and RHS is not.

Let us divide the set of reduced quadratic forms $ax^2 + bxy + cy^2$ of discriminant $-4n$ into seven categories
\begin{enumerate}
    \item Reduced forms which have $b \ \vdots \ 4$, $b > 0$, $a \centernot\vdots 2$, $c \centernot\vdots 2$,
    \item Reduced forms which have $b \ \vdots \ 2$ but $b \centernot\vdots 4$, $b > 0$, $a \centernot\vdots 2$, $c \ \vdots \ 2$,
    \item Reduced forms which have $b \ \vdots \ 2$ but $b \centernot\vdots 4$, $b > 0$, $a \ \vdots \ 2$, $c \centernot\vdots 2$,
     \item Reduced forms which have $b \ \vdots \ 4$, $b < 0$, $a \centernot\vdots  2$, $c \centernot\vdots 2$,
    \item Reduced forms which have $b \ \vdots \ 2$ but $b \centernot\vdots 4$, $b < 0$, $a \centernot\vdots 2$, $c \ \vdots \ 2$,
    \item Reduced forms which have $b \ \vdots \ 2$ but $b \centernot\vdots 4$, $b < 0$ $a \ \vdots \ 2$, $c \centernot\vdots 2$,
    \vspace{0.15cm}
    \item Reduced forms which have $b = 0$.
\end{enumerate}

First, we should justify this division. From equation $b^2 - 4ac = -4n = -16m - 4$ it follows that $b \ \vdots \ 2$. In case if $b \ \vdots \ 2$ but $b \centernot\vdots 4$, the product $ac$ should be even because otherwise $b^2 - 4ac$ is divisible by $8$. If both $a$ and $c$ are even then $b^2 - 4ac \equiv 4 \pmod{16}$, however $-16m - 4 \equiv 12 \pmod{16}$. If $b \ \vdots\ 4$ then $ac$ is odd because otherwise $b^2 - 4ac$ is divisible by $8$.

Just as in Сase $1$, we build the mapping separately for each of the categories of solution triples. We use the trick \ref{equation:trick} and equation $(2s - 1 + r)(2t - 1 + r) = n + r^2$ to show that these quadratic forms have discriminant $-4n$:
\begin{enumerate}
    \item $a = 2s - 1 + r$, $b = 2r$, $c = 2t - 1 + r$; \\
    $b^2 - 4ac = 4r^2 - 4(2s - 1 + r)(2t - 1 + r) = -4n$,
    \item $a = 2t - 1 + r$, $b  = 4t - 2$, $c = 2t + 2s - 2$; \\ $b^2 - (2a)(2c) = (2a - b)^2 - (2a)(2a + 2c - 2b) = 4r^2 - 4(2t - 1 + r)(2s - 1 + r) = -4n$,
    \item $a = 2t + 2s - 2$, $b = 4s - 2$, $c = 2s - 1 + r$; \\ $b^2 - (2a)(2c) = (2c - b)^2 - (2c)(2a + 2c - 2b) = 4r^2 - 4(2s - 1 + r)(2t - 1 + r) = -4n$,
    
    \item $a = 2s - 1 + r$, $b  = 2 - 4s$, $c = 2t + 2s - 2$; \\ $b^2 - (-2a)(-2c) = (-2a - b)^2 - (-2a)(-2a - 2c - 2b)  = 4r^2 - 4(2s - 1 + r)(2t - 1 + r) =-4n$, 
    \item $a = 2t + 2s - 2$, $b = 2 - 4t$, $c = 2t - 1 + r$; \\ $b^2 - (-2a)(-2c) = (-2c - b)^2 - (-2c)(-2a - 2c - 2b) = 4r^2 - 4(2t - 1 + r)(2s - 1 + r) = -4n$,
     \item $a = 2t - 1 + r$, $b = -2r$, $c = 2s - 1 + r$; \\
    $b^2 - 4ac = 4r^2 - 4(2t - 1 + r)(2s - 1 + r) = -4n$.
\end{enumerate}

Let us verify that these quadratic forms are indeed reduced. That is, in each case we need to check that $|b| \leq a \leq c$. In addition, if $|b| = a$ or $a = c$ then $b$ has to be non-negative.
\begin{enumerate}
    \item $2r \leq 2s - 1 + r \leq 2t - 1 + r$ is equivalent to $r \leq 2s - 1 \leq 2t - 1$;  since $b > 0$, the second condition is
satisfied,
 \item $4t - 2 \leq 2t - 1 + r \leq 2t + 2s - 2$ is equivalent to $2t - 1 \leq r \leq 2s - 1$;  since $b > 0$, the second condition is
satisfied, 
    \item $4s - 2 \leq 2t + 2s - 2 < 2s - 1 + r$ is equivalent to $2s - 1 \leq 2t - 1 \leq r$;  since $b > 0$, the second condition is
satisfied, 
   
        \item $4s - 2 < 2s - 1 + r < 2t + 2s - 2$ is equivalent to $2s - 1 < r < 2t - 1$; since $|b| \neq a$, $a \neq c$, the second condition is satisfied, 
    \item $4t - 2 < 2t + 2s - 2 < 2t - 1 + r$ is equivalent to $2t - 1 < 2s - 1 < r$; since $|b| \neq a$, $a \neq c$, the second condition is satisfied,
     \item $2r < 2t - 1 + r < 2s - 1 + r$ is equivalent to $r < 2t - 1 < 2s - 1$; since $|b| \neq a$, $a \neq c$, the second condition is satisfied.
\end{enumerate}

Notice that the introduced transform maps a solution triple $(r, s, t)$ of category $(i)$ onto a reduced quadratic form of category $(i)$.
Every  quadratic form in category $(i), \ i < 7$ has a preimage and any two solution triples in the same category have distinct images under this transform, since, clearly, for each category there exists an inverse transform. Therefore, this mapping is a bijection onto the set of reduced quadratic forms of categories $(1) - (6)$. 

To prove the claim of the theorem, it only remains to show that the size of category $(7)$ equals $\frac{1}{2}\sigma_0(n) + \frac{1}{2} \cdot \chi(n \ \text{is a square})$. Notice that any quadratic form with $b = 0$ corresponds to a pair of numbers $a$ and $c$ such that $a \leq c$ and $ac = n$. It is easy to see that the number of such pairs equals $\frac{1}{2}\sigma_0(n)$ if $n$ is not a square and $\frac{1}{2}(\sigma_0(n) + 1)$ otherwise.

\medskip

\textbf{Case 3:} $n = 8m + 3$.

\medskip

 A multiple of $x^2 + y^2$ cannot have the discriminant of $-n$. If this was the case, then $n$ would have had the form of $k^2$; however $n \equiv 3 \pmod{4}$ and $k^2 \not\equiv 3 \pmod{4}$. Also, a multiple of $x^2 + xy + y^2$ has the discriminant of $-n$ iff $n$ is divisible by $3$ and $\frac{n}{3}$ is a square number. By definition of Hurwitz class number, this quadratic form has weight $\frac{1}{3}$. Thus, $H(n)$ equals the number of reduced quadratic forms of discriminant $-n$ if $\frac{n}{3}$ is not a square and it is $\frac{2}{3}$ less than that otherwise. Similarly, $H(4n)$ equals the number of reduced quadratic forms of discriminant $-4n$ if $\frac{n}{3}$ is not a square and it is $\frac{2}{3}$ less than that otherwise.

If $n = 8m + 3$, $r$ could be even and odd. Let us consider these cases separately.

\medskip

\textbf{Case 3a:} $r$ is an even number.

\medskip
Let us divide the set of reduced quadratic forms of discriminant $-4n$ into categories:
\begin{enumerate}
    \item Reduced forms which have $b \ \vdots \ 4$, $b > 0$, $a \centernot\vdots 2$, $c \centernot\vdots 2$,
    \item Reduced forms which have $b \ \vdots \ 2$ but $b \centernot\vdots 4$, $b > 0$, $a \centernot\vdots 2$, $c \ \vdots \ 4$,
    \item Reduced forms which have $b \ \vdots \ 2$ but $b \centernot\vdots 4$, $b > 0$, $a \ \vdots \ 4$, $c \centernot\vdots 2$, 
    \item Reduced forms which have $b \ \vdots \ 2$ but $b \centernot\vdots 4$, $b < 0$, $a \centernot\vdots 2$, $c \ \vdots \ 4$, 
     \item Reduced forms which have $b \ \vdots \ 2$ but $b \centernot\vdots 4$, $b < 0$, $a \ \vdots \ 4$, $c \centernot\vdots 2$, 
    \item Reduced forms which have $b \ \vdots \ 4$, $b < 0$, $a \centernot\vdots 2$, $c \centernot\vdots 2$,
    \item Reduced forms which have $b \ \vdots \ 2$, but $b \centernot\vdots 4$, $a \ \vdots \ 2$, $c \ \vdots \ 2$,
   \vspace{0.15cm}
     \item Reduced forms which have $b = 0$.
\end{enumerate}
First, we should justify this division. From the equation $b^2 - 4ac = -4n = -32m - 12$ it follows that $b \ \vdots \ 2$. In case if $b \ \vdots \ 4$ the product $ac$ should be odd because otherwise $b^2 - 4ac$ is divisible by $8$. If $b \ \vdots \ 2$ but $b \centernot\vdots 4$, $ac$ is divisible by $4$. This is because since $\left(\frac{b}{2}\right)^2 \equiv 1 \pmod{4}$ and $-8m - 3 \equiv 1 \pmod{4}$, it follows that $ac \equiv 0 \pmod{4}$. Categories $(2) - (5)$ and $(7)$ thus describe all possibilities for $a$ and $c$ if $b \centernot\vdots 4$.

Consider the same mapping as in Case $2$. It is easy to see that a solution triple of category $(i)$ is mapped onto a reduced quadratic form of category $(i)$. Since $r$ is even, there is no solution triple with $r = 2s - 1 = 2t - 1$, and therefore the categories of solution triples do not intersect. For the same reasons as in Case $2$, this mapping is a bijection between the set of solution triples $(r, s, t)$ and the set of reduced quadratic forms of categories $(1) - (6)$.

Let us find the size of categories $(7)$ and $(8)$. If $ax^2 + bxy + cx^2$ is a reduced quadratic form in category $(7)$, then $\frac{a}{2}x^2 + \frac{b}{2}xy + \frac{c}{2}y^2$  is a reduced quadratic form of discriminant $-n$. On the other hand, since $n = 8m + 3$ any reduced quadratic form of discriminant $-n$ has odd second coefficient. Therefore, if we multiply any quadratic form of discriminant $-n$ by 2, we get a quadratic form from category $(7)$, which means that the size of category $(7)$ equals the number of reduced quadratic forms of discriminant $-n$, which is $H(n) - \frac{2}{3}\cdot[\frac{n}{3} \ \text{is a square number}]$. Notice that any quadratic form with $b = 0$ corresponds to a pair of numbers $a$ and $c$ such that $a \leq c$ and $ac = n$.  Since $n$ is not a square, the number of such pairs equals $\frac{1}{2}\sigma_0(n)$.

Observe that the number of reduced quadratic forms of discriminant $-4n$ is $H(4n) - \frac{2}{3}\cdot[\frac{n}{3} \ \text{is a square number}]$. Therefore, the number of solution triples $(r, s, t)$ with even $r$ equals $H(4n) - \frac{2}{3}\cdot[\frac{n}{3} \ \text{is a square number}]$ minus the sizes of categories $(7)$ and $(8)$. Recall identity \ref{equation:delta-3mod8} to simplify the formula:

\begin{equation*}
H(4n) - H(n) - \frac{1}{2}\sigma_0(n) = 3H(n) - \frac{1}{2}\sigma_0(n).
\end{equation*}

\medskip

\textbf{Case 3b:} $r$ is an odd number.

To conclude the proof, it only remains to show that the number of solution triples with odd $r$ equals $3H(n)$. As usual, we divide all solution triples $(r, s, t)$ except for $r = 2s - 1 = 2t - 1$ (if it exists) into six categories. For each of the categories we define an injection onto the set of reduced quadratic forms of discriminant $-n$:
\begin{enumerate}
    \item $a = s + \frac{r - 1}{2}$, $b = r$, $c = t + \frac{r - 1}{2}$; \\
    $b^2 - 4ac = r^2 - (2s - 1 + r)(2t - 1 + r) = -n$,
      \item $a =  t + \frac{r - 1}{2}$, $b  = 2t - 1$, $c = s + t - 1$; \\ $b^2 - (2a)(2c) = (2c - b)^2 - (2c)(2a + 2c - 2b) = r^2 - (2s - 1 + r)(2t - 1 + r) = -n$, 
    \item $a = t + s - 1$, $b = 2s - 1$, $c = s + \frac{r - 1}{2}$; \\
    $b^2 - (2a)(2c) = (2a - b)^2 - (2a)(2a + 2c - 2b) = r^2 - (2s - 1 + r)(2t - 1 + r) = -n$, 
      \item $a =  s + \frac{r - 1}{2}$, $b  = 1 - 2s$, $c = s + t - 1$ only if $r, s, t$ are distinct; \\ $b^2 - (2a)(2c) = (2a - b)^2 - (2a)(2a + 2c - 2b) = r^2 - (2s - 1 + r)(2t - 1 + r) = -n$, 
    \item $a = t + s - 1$, $b = 1 - 2t$, $c = s + \frac{r - 1}{2}$ only if $r, s, t$ are distinct; \\
    $b^2 - (2a)(2c) = (2c - b)^2 - (2c)(2a + 2c - 2b) = r^2 - (2s - 1 + r)(2t - 1 + r) = -n$,
    \item $a = t + \frac{r - 1}{2}$, $b = -r$, $c = s + \frac{r - 1}{2}$; \\
    $b^2 - 4ac = r^2 - (2s - 1 + r)(2t - 1 + r) = -n$.
    
\end{enumerate}

 Let us verify that these quadratic forms are indeed reduced. For the first three categories, since $b > 0$ for all of them, we just need to show that $b \leq a \leq c$. For the latter three categories, since $b < 0$, we need to show that $|b| < a < c$.
\begin{enumerate}
     \item $r \leq s + \frac{r - 1}{2} \leq t + \frac{r - 1}{2}$ is equivalent to $r \leq 2s - 1\leq 2t - 1$,
    \item $2t - 1 \leq t + \frac{r - 1}{2} \leq t + s - 1$ is equivalent to $2t - 1 \leq r \leq 2s - 1$,
    \item $2s - 1 \leq t + s - 1 \leq s + \frac{r - 1}{2}$ is equivalent to $2s - 1 \leq 2t - 1 \leq r$,

    \item $2s - 1 < s + \frac{r - 1}{2} < t + s - 1$ is equivalent to $2s - 1 < r < 2t - 1$,
    \item $2t - 1 < s + t - 1 < s + \frac{r - 1}{2}$ is equivalent to $2t - 1 < 2s - 1 < r$,
    \item $r < t + \frac{r - 1}{2} < s + \frac{r - 1}{2}$ is equivalent to $r < 2t - 1 < 2s - 1$.
\end{enumerate}

Now, consider the union of the six injections above. The domain of this transform consists of all solution triples except for the one with $r = 2t - 1 = 2s - 1$, if it exists. It is easy to verify that the quadratic form $(a, b, c)$ with $a = b = c$, if it exists, is excluded from the range of this transform. The injections have simple linear inverse transformations. The set of preimages of a reduced quadratic form $(a, b, c)$ under this transformation can be easily found by applying the inverse transformations to $(a, b, c)$. As it turns out, any quadratic form $(a, b, c)$, except for the one with $a = b = c$, has exactly three preimages. Indeed, $b \neq 0$ because $n$ is odd. Since three of the injections yield quadratic forms with positive $b$ coefficients and three yield quadratic forms with negative $b$ coefficients, any $(a, b, c)$ lies in domain of exactly three of the six inverse transformations. Finally, it is easy to see that any two of the inverse transformations send a quadratic form to different solution triples.

If $(a, b, c) = (z, z, z)$ for some $z$ exists then $z^2 - 4z^2 = -n$, and $ax^2 + bxy + cy^2$ is a multiple of $x^2 + xy + y^2$. This happens iff there exists a solution triple with $r = 2s - 1 = 2t - 1$. In this case, we add this solution triple to the domain of our transformation and define the image of it to be $(z, z, z)$. We have thus built a surjective mapping for which any triple $(a, b, c)$ has exactly three times more preimages than its weight in $H(n)$. Thus, the number of solution triples $(r, s, t)$ with odd $r$ indeed equals $3H(n)$.

\medskip

\textbf{Case 4:} $n = 8m + 7$.

\medskip

The number of reduced quadratic forms of discriminant $-n$ equals $H(n)$ because $n = 8m + 7$ cannot be represented as $3k^2$ or $k^2$ and thus no multiple of $x^2 + xy + y^2$ or $x^2 + y^2$ has discriminant $-n$. Indeed, $8m + 7 \equiv 7 \pmod{8}$ but $k^2 \equiv 1 \pmod{8}$, $3k^2 \equiv 3 \pmod{8}$ if $k$ is odd. Similarly, the number of reduced quadratic forms of discriminant $-4n$ equals $H(4n)$. 

Let us divide the solution triples $(r, s, t)$ into four categories

\begin{enumerate}
    \item $r \ \vdots \ 2$, $s - t \centernot\vdots 2$, 
    \item $r \centernot\vdots 2$, $2s - 1 + r \ \vdots \ 4$, $2t - 1 + r \ \vdots \ 4$,
    \item $r \centernot\vdots 2$, $2s - 1 + r \centernot\vdots 4$, $2t - 1 + r \ \vdots \ 4$,
    \item $r \centernot\vdots 2$, $2t - 1 + r \centernot\vdots 4$, $2s - 1 + r \ \vdots \ 4$.
\end{enumerate}
Let us justify this division. It is easy to see that the categories do not intersect. Consider the equation $r^2 + n = (2s - 1 + r)(2t - 1 + r)$. For any solution triple $(r, s, t)$ with $r \ \vdots \ 2$ $r^2 + n \equiv 3 \pmod{4}$, therefore $2s - 1 + r \not\equiv 2t - 1 + r \pmod{4}$, which implies that $s - t \centernot\vdots 2$. From the same equation, if $r \centernot\vdots 2$ then $(2s - 1 + r)(2t - 1 + r) \ \vdots \ 8$. Therefore, $2s - 1 + r$ and $2t - 1 + r$ are both even and at least one of them should be divisible by $4$. Categories $(2) - (4)$ thus describe all possibilities for $2s - 1 + r$ and $2t - 1 + r$. 

We need to interpret the factor $(-1)^{r+s+t}$. It is easy to see that $r + s + t \centernot\vdots 2$ for all solution triples in category $(1)$. For all solution triples in category $(2)$ it holds that $2s - 2t \ \vdots \ 4$, this $s - t \ \vdots \ 2$ and $r + s + t \centernot\vdots 2$. And since $2s - 2t \centernot\vdots 4$ for all solution triples in categories $(3)$ and $(4)$, $s - t \centernot\vdots 2$ and $r + s + t \ \vdots \ 2$. Therefore, to calculate the LHS of the desired equation we need to subtract the size of the union of subsets $(1)$ and $(2)$ from the size of the union of subsets $(3)$ and $(4)$. We will achieve this by calculating sizes of each of the categories separately.

Recall the case of even $r$ for $n = 8m + 3$. We can repeat that argument for the case of even $r$ for $n = 8m + 7$. Just as in case of even $r$ for $n = 8m + 3$ we get that
\begin{equation*}
    H(4n) - H(n) - \frac{1}{2}\sigma_0(n) = H(n) - \frac{1}{2}\sigma_0(n).
\end{equation*}
The equation above holds due to identity \ref{equation:delta-7mod8}.

Recall the case of odd $r$ for $n = 8m + 3$. We can repeat that argument for the case of odd $r$ for $n = 8m + 7$. Just as in case of $n = 8m + 3$, the number of solution triples with odd $r$ equals $3H(n)$. From this argument it follows that every reduced quadratic form $(a, b, c)$ has exactly three preimages. If we prove that any two of the preimages come from different categories $(2) - (4)$, then we get that the sizes of categories $(2) - (4)$ are equal to each other and $H(n)$. 

Without loss of generality, assume that $b > 0$. Let us express $2s - 1 + r$ and $2t - 1 + r$ in terms of $a, b, c$ for each of the three preimages
\begin{enumerate}
    \item $a = s + \frac{r - 1}{2}$, $b = r$, $c = t + \frac{r - 1}{2}$; $2s - 1 + r = 2a, \ 2t - 1 + r = 2c$;
    \item $a =  t + \frac{r - 1}{2}$, $b  = 2t - 1$, $c = s + t - 1$; $2s - 1 + r = 2a + 2c - 2b$, $2t - 1 + r = 2a$;
    \item $a = t + s - 1$, $b = 2s - 1$, $c = s + \frac{r - 1}{2}$; $2s - 1 + r = 2c$, $2t - 1 + r = 2a + 2c - 2b$.
\end{enumerate}

Notice that in every instance $2s - 1 + r$ and $2t - 1 + r$ are equal to $2a$, $2c$ or $2a + 2c - 2b$. Since $b^2 - 4ac = -n = -8m - 7$, it follows that $4ac \ \vdots \ 8$, therefore at least one of $2a$ and $2c$ is divisible by $4$. And since $b \centernot\vdots 2$, it follows that exactly one out of $2a$, $2c$, $2a + 2c - 2b$ is not divisible by $4$. Thus, for exactly one of the above preimages both $2s - 1 + r$ and $2t - 1 + r$ are divisible by $4$; for exactly one of them $2s - 1 + r \centernot\vdots 4$ and $2t - 1 + r \ \vdots \ 4$; for exactly one of them $2t - 1 + r \centernot\vdots 4$ and $2s - 1 + r \ \vdots \ 4$. This is exactly what we needed to prove.

We have found sizes of each of the categories, and we can conclude that
\begin{equation*}\sum_{\substack{r, s, t \geq 1 \\ (2s - 1 +r)(2t - 1+r) = n + r^2}} (-1)^{r + s + t} = H(n) + H(n) - \left(H(n) + H(n) - \frac{1}{2}\sigma_0(n)\right) = \frac{1}{2}\sigma_0(n).\qedhere
\end{equation*}

\end{proof}

\section{Concluding remarks and questions}\label{section:conclusion}

The methods that we develop here appear related to those of Mordell \cite{Mord1}, where he counts non-negative solutions to (\ref{equation:HLM}).  However, our implementation is different than what one finds in \cite{Mord1}.  For example, we have to work with a weight term of $(-1)^{r+s+t}$, and our motivation comes from studying limiting cases  of a higher-dimensional Kronecker-type identity.

We finish with some questions.
\begin{enumerate}
    \item Do other specializations of Theorem \ref{theorem:result}, or Theorems $1.3$ and $1.4$ of \cite{Mo2017A}, allow us to count solutions to other quadratic forms?
    \item In the specializations of Theorem \ref{theorem:result} that we have carried out here and in \cite{Mo2017B}, the four summands on the right-hand side always reduces to a single quotient of theta functions.  In Section \ref{section:newAC} there is cancellation, and in \cite[Sections 3, 4]{Mo2017B} the four summands are scalar multiples of each other.  Does anything interesting happen when this is not the case?

\end{enumerate}
Also note that we did not actually obtain the generating function for the number of solutions to $n=x^2+2y^2+2z^2$, which is given by
\begin{equation*}
        \sum_{n=0}b(n)q^{n}=j(-q;q^2)j(-q^2;q^4)^2.
\end{equation*}
Instead, we obtained
\begin{equation*}
        \sum_{n=0}a(n)q^{n}=j(q;q^2)j(q^2;q^4)^2,
\end{equation*}
where the equality $|a(n)|=b(n)$ is easily seen by using (\ref{equation:j-split}) with $m=2$.


\begin{thebibliography}{999999}


\bibitem{A} G. E. Andrews, {\em Ramanujan's fifth and seventh order mock theta functions}, Trans. Amer. Math. Soc. {\bf 293}, (1986), no. 1,  113--134.

\bibitem{Co} H. Cohen, A Course in Computational Algebraic Number Theory, Springer-Verlag Berlin Heidelberg, 1993.

\bibitem{Cr} R. E. Crandall, {\em New Representations for the Madelung constant}, Experiment. Math. {\bf 8} (1999), no. 4, 367--379. 

\bibitem{G} C. F. Gauss, {\em Disquisitiones Arithmeticae}, Leipzig, 1801.

\bibitem{HM} D. R. Hickerson, E. T. Mortenson, {\em Hecke-type double sums, Appell--Lerch sums, and mock theta functions, I}, Proc. Lond. Math. Soc. (3) {\bf 109} (2014), no. 2, 382--422. 

\bibitem{Kron1} L. Kronecker, {\em Zur Theorie der elliptischen Functionen}, Monatsber. K. Akad. Wiss. zu Berlin (1881), 1165--1172.

\bibitem{Kron2} L. Kronecker, {\em Leopold Kronecker's Werke}, Bd. IV, B.G. Teubner, Leipzig, 1929, reprinted by Chelsea, New York, 1968.



\bibitem{Mord1} L. J. Mordell, {\em On the number of solutions in positive integers of the equation $yz+zx+xy=n$}, Amer. J. Math. Vol. {\bf 45}, No. 1 (Jan. 1923), 1--4.

\bibitem{Mo2017A} E. T. Mortenson, {\em A double-sum Kronecker-type identity}, 
Adv. in Appl. Math. {\bf 82} (2017), 155--177.

\bibitem{Mo2017B} E. T. Mortenson, {\em A Kronecker-type identity and the representations of a number as a sum of three squares}, Bull. Lond. Math. Soc. {\bf 49} (2017), no. 5, 770--783.


\bibitem{Weil} A. Weil, {\em Elliptic Functions According to Eisenstein and Kronecker}, Springer-Verlag, Berlin, 1976.


\bibitem{Warn} S. O. Warnaar, {\em Ramanujan's $_{1}\psi_1$ summation}, in K. Alladi, ed., Srinivisa Ramanujan: going strong at 125, Part II.  Notices Amer. Math. Soc. {\bf 60}, no. 1 (2013), 10--23.

\bibitem {Za} D. B. Zagier, {\em Zetafunktionen und quadratische K\"orper:  eine Einf\"uhrung in die h\"ohere Zahlentheorie}, Springer-Verlag, Berlin, Heidelberg, New York, $1981$.





\end{thebibliography}
\end{document}